\documentclass{amsart}
\usepackage{amsmath,amsfonts,amssymb,stmaryrd}
\usepackage{graphicx,psfrag,subfigure}

\newtheorem{theorem}{Theorem}[section]

\newtheorem{lemma}[theorem]{Lemma}
\newtheorem{prop}{Proposition}

\theoremstyle{definition}
\newtheorem{definition}[theorem]{Definition}
\newtheorem{rk}{Remark}

\newtheorem{example}{Example}

\newcommand{\N}{\mbox{$\mathbb{N}$}}

\numberwithin{equation}{section}

\begin{document}

\title[ Sensitive actions in non-compact spaces. ]{Sensitive actions in non-compact spaces. }

\author[Jorge Iglesias ]{}

% It is required to enter 2010 MSC.
%\subjclass{Primary: 37B05; Secondary: 37B10, 54H20, 58F03.}
% Please provide minimum  5 keywords.
 %\keywords{semigroups actions, transitivity, sensitivity.}

%\address{N. Guelman, Universidad de La Rep\'ublica. Facultad de Ingenieria. IMERL. Julio
%Herrera y Reissig 565. C.P. 11300. Montevideo, Uruguay }
%\email{nguelman@fing.edu.uy}
%
%\address{J. Iglesias, Universidad de La Rep\'ublica. Facultad de Ingenieria. IMERL. Julio
%Herrera y Reissig 565. C.P. 11300. Montevideo, Uruguay}
%\email{jorgei@fing.edu.uy }

%\address{A. Portela, Universidad de La Rep\'ublica. Facultad de Ingenieria. IMERL. Julio
%Herrera y Reissig 565. C.P. 11300. Montevideo, Uruguay }
%\email{aldo@fing.edu.uy }

\maketitle

\centerline{\scshape  Jorge Iglesias $^*$  and Aldo Portela$^*$}
\medskip
{\footnotesize
 \centerline{Universidad de La Rep\'ublica. Facultad de Ingenieria. IMERL}
   \centerline{ Julio Herrera y Reissig 565. C.P. 11300}
   \centerline{ Montevideo, Uruguay}}

\bigskip

 %\centerline{(Communicated by )}

\begin{abstract}
Devaney defines a function as chaotic if it satisfies the following three conditions: transitivity, having a dense set of periodic points, and sensitive dependence on initial conditions. In \cite{3}, it was demonstrated that the first two conditions imply the third. This result was generalized in \cite{aak} by replacing the density of periodic points with the density of minimal points. The result was further generalized in \cite{g} for group actions, in \cite{km} for $C$-semigroups actions, and in \cite{d} for a continuous semi-flow with $X$ being a Polish space. Subsequently, in \cite{ip1} and \cite{ip2}, it was generalized for compact spaces and for non-compact spaces in \cite{z}. The objective of this work is to generalize the result in \cite{z}, providing a simple proof.
\end{abstract}

\section{Introduction.}

In the study of dynamical systems, the concept of chaos plays a pivotal role in understanding complex and unpredictable behaviors. A seminal contribution to this field was made by Devaney in \cite{de}, who proposed a formal definition of a chaotic map. This definition has since become foundational in the analysis of chaotic functions.

\begin{definition}
\label{def1}
Let \( X \) be a metric space. A continuous function \( f: X \to X \) is called chaotic if it satisfies the following three conditions:
\begin{enumerate}
    \item \( f \) is topologically transitive.
    \item The set of periodic points of \( f \) is dense.
    \item \( f \) exhibits sensitive dependence on initial conditions.
\end{enumerate}
\end{definition}

A function \( f \) is said to be topologically transitive if for every pair of non-empty open sets \( U \) and \( V \) in \( X \), there exists an \( n \in \mathbb{N} \) such that \( f^n(U) \cap V \neq \emptyset \).

A point \( p \in X \) is called a fixed point if \( f(p) = p \). A point \( p \in X \) is called periodic if there exists a \( k \in \mathbb{N} \) (\( k \geq 1 \)) such that \( f^k(p) = p \). The period of \( p \) is defined as \( \min\{ k \geq 1: f^k(p) = p\} \).

We say that \( f \) is sensitive if there exists an \( \alpha > 0 \) such that for every open set \( U \) of \( X \) and for every \( x \in U \), there exist \( y \in U \) and \( k \in \mathbb{N} \) such that \( d(f^k(x), f^k(y)) > \alpha \). The number \( \alpha \) is called a sensitivity constant.

In \cite{3}, the following significant result was proven:

\begin{prop}
\label{prop_principal}
If \( f: X \to X \) is continuous, topologically transitive, and the set of periodic points of \( f \) is dense, then \( f \) is sensitive.
\end{prop}

Subsequent research has focused on generalizing this result. For instance, in \cite{aak}, the density of periodic points was replaced with the density of minimal points. A minimal set \( M \subset X \) is a non-empty set such that \( \overline{O(y)} = M \) for every \( y \in M \). A point \( x \in X \) is minimal if the set \( \overline{O(x)} \) is minimal. We denote the set of minimal points by \( \mathcal{M} \). The function \( f \) is called minimal if \( \overline{O(x)} = X \) for every \( x \in X \).

An illustrative example with a dense set of minimal sets but without periodic points is the function \( f: S^1 \times S^1 \to S^1 \times S^1 \) defined by \( f(x, y) = (x^2, R_I(y)) \), where \( R_I \) is an irrational rotation.

So far, we have considered the dynamical system generated by a function. This concept can be generalized as follows:

The dynamical system that we will consider is formally defined as a triplet \((\mathcal{S}, X, \Phi)\), where \(\mathcal{S}\) is a topological semigroup and \(\Phi: \mathcal{S} \times X \to X\) is a continuous function with \(\Phi(s_1, \Phi(s_2, x)) = \Phi(s_1 s_2, x)\) for all \(s_1, s_2 \in \mathcal{S}\) and for all \(x \in X\). Here, \(X\) is a metric space, which is infinite and has no isolated points.

The map \(\Phi\) is called an action of \(\mathcal{S}\) on \(X\). We will denote \(\Phi(s, x)\) by \(\Phi_s(x)\).

For any \(x \in X\), we define the orbit of \(x\) as \(O(x) = \{\Phi_s(x): s \in \mathcal{S}\}\). Sometimes, the set \(O(x)\) is denoted by \(\mathcal{S}x\). A non-empty set \(Y \subset X\) is minimal if \(\overline{O(y)} = Y\) for any \(y \in Y\).

A point \(x \in X\) is minimal if the set \(\overline{O(x)}\) is a minimal set. We denote by \(\mathcal{M}\) the set of minimal points. We say that the dynamical system \((\mathcal{S}, X, \Phi)\) is minimal if there exists a minimal point \(x \in X\) such that \(\overline{O(x)} = X\).

The dynamical system \((\mathcal{S}, X, \Phi)\) is point transitive (PT) if there exists \(x \in X\) such that \(\overline{O(x)} = X\). It is topologically transitive (TT) if, given two open and non-empty sets \(U, V \subset X\), there exists \(s \in \mathcal{S}\) such that \(\Phi_s(U) \cap V \neq \emptyset\). It is densely point transitive (DPT) if there exists a dense set \(Y \subset X\) of transitive points. Denote by \(Trans(X)\) the set of transitive points and by \(Trans(X)^*\) the set \(\{x \in Trans(X): \Phi_s(x) \in Trans(X), \forall s \in \mathcal{S}\}\).

If \(X\) is a Polish space (i.e., separable completely metrizable topological space) or \(X\) is a locally compact metric space with a countable base, then TT implies DPT.

We say that a dynamical system \((\mathcal{S}, X, \Phi)\) is sensitive if there exists \(\alpha > 0\) such that for all \(x \in X\) and for all \(\delta > 0\), there exists \(y \in B(x, \delta)\) and \(s \in \mathcal{S}\) such that \(d(\Phi_s(x), \Phi_s(y)) > \alpha\).

%Further generalizations have been explored in various contexts, such as group actions in \cite{g}, and for \(C\)-semigroups in \cite{km}.
 We denote by $Ss_0=\{ss_0: \ s\in S \}$. A \(C\)-semigroup \(S\) is a semigroup such that \(S \setminus Ss_0\) is relatively compact. A point \(x\) is called almost periodic if $O(x)$ is minimal and compact. If the set of almost periodic points is dense in \(X\), we say that \((\mathcal{S}, X, \Phi)\) satisfies the Bronstein condition. If, in addition, the system \((\mathcal{S}, X, \Phi)\) is topologically transitive (TT), we say that it is an \(M\)-system.

In \cite[Theorem 5.5]{km}, the following result was proven:

\begin{theorem}
\label{thm:5.7}
Let \((\mathcal{S}, X, \Phi)\) be a dynamical system, where \((X, d)\) is a Polish \(S\)-system and \(S\) is a \(C\)-semigroup. If \((\mathcal{S}, X, \Phi)\) is an \(M\)-system which is not minimal or not equicontinuous, then \((\mathcal{S}, X, \Phi)\) is sensitive.
\end{theorem}

In \cite{ip1}, assuming that \(X\) is a compact set, the previous result is generalized:

\begin{theorem}
\label{thm:1.1}
Assume that the dynamical system \((\mathcal{S}, X, \Phi)\) is TT, non-minimal, and \(\mathcal{M}\) (the set of minimal points) is a dense set. Then \((\mathcal{S}, X, \Phi)\) is sensitive.
\end{theorem}

We say that $\phi_s\in \Phi$  is  almost open if $int (\phi_s (U))$ ($int(W)$ denote interior of $W$) is nonempty whenever $U$ is opene ( open and nonempty). 
%This is equivalent to the condition that $\phi_s ^{-1} (D)$ is
%dense in $X$ whenever $D$ is dense in $X$. 
The system \((\mathcal{S}, X, \Phi)\) is almost open if $\phi_s$ is  almost open for any $s\in \mathcal{S}$.
The class of dynamical systems \((\mathcal{S}, X, \Phi)\) defined by the union of C-semigroups and almost open semigroups $\mathcal{S}$ is denoted by $\mathfrak{U}$  

In \cite{z}, the following theorem was established for non-compact spaces:

\begin{theorem}
\label{thm:3}
Let \(X\) be a locally compact metric space with a countable base. Let \((\mathcal{S}, X, \Phi)\in \mathfrak{U}\) satisfy the following two conditions:
\begin{enumerate}
    \item It is topologically transitive.
    \item At least one of the sets,
    \begin{enumerate}
        \item the union of minimal sets,
        \item the union of points having closed orbits,
    \end{enumerate}
    is a proper dense subset in \(X\).
\end{enumerate}
Then \((\mathcal{S}, X, \Phi)\) is sensitive to any metric \(\rho\) metrizing \(X\).
\end{theorem}

The following example shows the difference between (TT) and (PT):

\begin{example}
\label{example3}
Let \(S_1\) and \(S_2\) be two disjoint circles. Consider \(X = S_1 \cup S_2\) and functions \(f, g, h: X \to X\) such that:
\begin{itemize}
\item \(f|_{S_1}\) is an irrational rotation and \(f|_{S_2}\) is the identity,
\item \(g|_{S_1}\) is the identity and \(g|_{S_2}\) is an irrational rotation,
\item \(h|_{S_1}\) is the identity and \(h(S_2) \subset S_1\).
\end{itemize}
Let \(\mathcal{S}\) be the free semigroup generated by three elements \(\{a, b, c\}\) and \(\Phi\) the action generated by \(\Phi_a = f\), \(\Phi_b = g\), and \(\Phi_c = h\). The following properties are easy to prove:
\begin{enumerate}
\item The set of transitive points coincides with \(S_2\). It is worth noting that it is false that if \(x \in X\) is transitive, then \(\Phi_s(x)\) is transitive for all \(s \in \mathcal{S}\). Consequently, the set of transitive points is not dense.
\item For any pair of open sets \(U \subset S_1\), \(V \subset S_2\), there does not exist \(s \in \mathcal{S}\) such that \(\Phi_s(U) \cap V \neq \emptyset\). That is, \((\mathcal{S}, X, \Phi)\) is point transitive but not topologically transitive.
\end{enumerate}
\end{example}

\newpage
 We will now state our main results:

\begin{theorem}\label{theo0}
 Let $X$ be a locally compact metric space with a countable base and $( \mathcal{S}, X, \Phi )$ that satisfy
the following conditions:\\
(1) is PT,\\
(2) $\mathcal{M}$ is dense and\\
(3) non minimal.\\
Then $( \mathcal{S}, X, \Phi )$  is sensitive.
\end{theorem}

The difference between this result and Theorem~\ref{thm:3} is that it is not necessary for the dynamical system to belong to the family \(\mathfrak{U}\), and furthermore, (TT) is replaced by (PT).\\
The proof is inspired by the ideas used in \cite{z}.
Furthermore, we derive the following results that imply sensitivity.\\
Let  $\mathcal{F}$ be the union of points having closed orbits.
\begin{theorem}\label{theo1}
 Let $X$ be a locally compact metric space with a countable base and $( \mathcal{S}, X, \Phi )$ that satisfy
the following conditions:\\
(1) $Trans(X)^{*}\neq\emptyset$,\\
(2)  $\mathcal{F}$ is a dense subset in $X$ and\\
(3) non mininal\\ %there exists $x\in X$ such that $\overline{O(x)}\neq X$.\\
Then $( \mathcal{S}, X, \Phi )$  is sensitive.
\end{theorem}

If $Trans(X)^{*}\neq\emptyset$, we obtain a generalization of the Theorem \ref{theo0}.\\
 Denote by $\mathcal{M}^{-1}=\cup_{s \in \mathcal{S}}\Phi_s^{-1}(\mathcal{M})$.

\begin{theorem}\label{theo2}
 Let $X$ be a locally compact metric space with a countable base and $( \mathcal{S}, X, \Phi )$ that satisfy
the following conditions:\\
(1) $Trans(X)^{*}\neq\emptyset$;\\
(2)  $\mathcal{M}^{-1}$ is dense in $X$ and\\
(3) non minimal.\\
Then $( \mathcal{S}, X, \Phi )$  is sensitive.
  \end{theorem}

\section{Proof of Theorems \ref{theo0} and \ref{theo1}.}

For the proof of the main theorems we need the following lemma:

\begin{lemma}\label{transitive}
 Assume that $( \mathcal{S}, X, \Phi )$ is PT and $\mathcal{M}$ is dense. Then:
 \begin{enumerate} 
 \item $( \mathcal{S}, X, \Phi )$  is TT.
 \item  If $X$ is locally compact metric space with a countable base, $( \mathcal{S}, X, \Phi )$  is DPT.   
\end{enumerate}
\end{lemma}

%\begin{proof}

{\bf{Proof item (1).}} Let $U$ and $V $ be opene sets and $z\in X$ such that $\overline{O(z)}=X$. Then there exist $\Phi_s, \Phi_{s_{1}} $ and $W_z$ a neighbourhood of $z$ such that $\Phi_s (W_z)\subset U$ and $\Phi_{s_{1}} (W_z)\subset V$.
   As $\mathcal{M}$ is dense there exists $y\in \mathcal{M}\cap W_z$. Then $\Phi_s (y)\in U$ and $\Phi_{s_{1}} (y)\in V$. As $y$ is a minimal point then there exists $\Phi_{s_{2}} $ such that $\Phi_{s_{2}}\Phi_s (y)\in V$. Then $\Phi_{s_{2}}(U)\cap V\neq \emptyset$.\\
   {\bf{Proof item (2).}} Let $\{U_k\}_{k\in\N}$ be a countable base. By item 1), given $k,m\in \N$ there exists $s(k,m)\in  \mathcal{S}$ such that $\Phi_{s(k,m)} (U_m)\cap U_k\neq\emptyset$. So, by Baire's Theorem,
 
 $$R=\bigcap_{k\geq 0}\bigcup_{m\geq 0}  \Phi^{-1}_{s(k,m)} (U_k) \mbox{ is a residual of dense orbits}. $$
 
 % If $X$ is locally compact metric space with a countable base, by Baire Theorem, $( \mathcal{S}, X, \Phi )$  is DPT .
%\end{proof}
$\Box$

\vspace{.5cm}

It follows, from the definition of sensitivity, that:

\begin{rk}\label{rk1}
If  $( \mathcal{S}, X, \Phi )$ is nonsensitive then  for all $k\in \N$ there exists an opene set $V_{k}$ such that $diam (\Phi_s (V_{k})) \leq 1/k$ for all $s\in \mathcal{S}$.
\end{rk}

{\bf{Proof of Theorem \ref{theo0}:}}\\
Suppose that \((\mathcal{S}, X, \Phi)\) is not sensitive. Given \(x \in X\), we will prove that \(\overline{O(x)} = X\), which leads to a contradiction because \((\mathcal{S}, X, \Phi)\) is non-minimal.\\ Let $w\in X$ and consider $B(w,r)$, $r>0$. Let $k_0$ be such that $1/k_{0}< r/4$. As $( \mathcal{S}, X, \Phi )$ is nonsensitive, by Remark \ref{rk1}, there exists an opene $V_{k_{0}}$ such that $diam (\Phi_s (V_{k_{0}})) \leq 1/k{_{0}}$ for all $s\in \mathcal{S}$.\\
As $( \mathcal{S}, X, \Phi )$ is PT and $\mathcal{M}$ is dense, then (by Lemma \ref{transitive}) $( \mathcal{S}, X, \Phi )$  is DPT. So, there exists $z\in V_{k_{0}}$ be such that $\overline{O(z)}=X$.\\ Let $s_0\in  \mathcal{S}$ and $\varepsilon >0$ be such that 
\begin{equation}\label{eq1}
 B(z,\varepsilon )\subset V_{k_{0}} \mbox{ and } \Phi_{s_{0}}(B(z,\varepsilon ))\subset B(w,r/8). 
\end{equation}

As $\overline{O(z)}=X$, there exists a sequence $\{s^{'}_n\}$,   $s^{'}_n\in \mathcal{S}$ such that $\Phi_{s^{'}_{n}}(z)\to x$. As $\mathcal{M}$ is dense, there exists a minimal set $M_1$ and $y\in M_1$ such that $y\in B(z,\varepsilon )\subset V_{k_{0}}$. Without loss of generality we can assume that $\Phi_{s^{'}_{n}}(y)\to y_0$ (here we are using that $X$ locally compact). As $y_0\in M_1$ there exist $s_1\in \mathcal{S}$ such that $\Phi_{s_{1}}(y_0)\in  B(z,\varepsilon )$.
  
  Since $z,y\in V_{k_{0}}$ then $ d ( \Phi_{s_{0}} \Phi_{s_{1}} \Phi_{s^{'}_{n}}(z) , \Phi_{s_{0}} \Phi_{s_{1}} \Phi_{s^{'}_{n}}(y) )\leq 1/k{_{0}}$. As $\Phi_{s^{'}_{n}}(z)\to x$ and $\Phi_{s^{'}_{n}}(y)\to y_0$ we obtain 

\begin{equation}\label{eq2}
  d ( \Phi_{s_{0}} \Phi_{s_{1}}(x) , \Phi_{s_{0}} \Phi_{s_{1}} (y_0) ) = \lim_{n\to +\infty} d ( \Phi_{s_{0}} \Phi_{s_{1}} \Phi_{s^{'}_{n}}(z) , \Phi_{s_{0}} \Phi_{s_{1}} \Phi_{s^{'}_{n}}(y) )  \leq  1/k{_{0}}\cdot
  \end{equation}

As $d(w,\Phi_{s_{0}} \Phi_{s_{1}} (y_0) )\leq r/8$ and $d ( \Phi_{s_{0}} \Phi_{s_{1}}(x) , \Phi_{s_{0}} \Phi_{s_{1}} (y_0) )\leq  1/k{_{0}}<r/4$ then $d(w,\Phi_{s_{0}} \Phi_{s_{1}} (x) )< r$.
  
$\Box$

\vspace{.5cm}

{\bf{Proof of Theorem \ref{theo1}:}}\\
Suppose that $( \mathcal{S}, X, \Phi )$ is nonsensitive.
  By item(3) of the hypothesis, there exists \( x \in X \) such that \( \overline{O(x)} \neq X \). We will proceed to show that \( \overline{O(x)} = X \), which leads to a contradiction. Let $w\in X$ and consider $B(w,r)$, $r>0$. Let $k_0$ be such that $1/k_{0}< r/4$. As $( \mathcal{S}, X, \Phi )$ is nonsensitive, by Remark \ref{rk1}, there exists an opene $V_{k_{0}}$ such that $diam (\Phi_s (V_{k_{0}})) \leq 1/k{_{0}}$ for all $s\in \mathcal{S}$.

 Let $z\in Trans(X)^{*}\cap V_{k_{0}}$ be  and a sequence $\{s^{'}_n\}$, $s^{'}_n\in \mathcal{S}$, such that $\Phi_{s^{'}_{n}}(z)\to x$.  
 
 As  $\mathcal{F}$  is dense, there exists $y\in  V_{k_{0}}\cap  \mathcal{F} $.   Without loss of generality we can assume that $\Phi_{s^{'}_{n}}(y)\to y_0$ (here we are using that $X$ locally compact). As $y\in \mathcal{F}$, then there existe $s_0$ such that $y_0= \Phi_{s_{0}}(y).$
 
 As $z\in Trans(X)^{*}$ there exists  $s_1$ such that $\Phi_{s_{1}}\Phi_{s_{0}}(z)\in B(w,r/8)$. Since $y,z\in V_{k_{0}}$ then:\\ 
 
 (1)  $d ( \Phi_{s_{1}} \Phi_{s_{0}}(z) , \Phi_{s_{1}} \Phi_{s_{0}} (y) )\leq  1/k{_{0}}< r/4$.\\

 (2) As $y_0= \Phi_{s_{0}}(y)$ then  $d (  \Phi_{s_{1}} \Phi_{s_{0}} (z),\Phi_{s_{1}} (y_0)  )\leq  1/k{_{0}}<r/4$.\\
 
 (3) $d ( \Phi_{s_{1}} (y_0)  , \Phi_{s_{1}} (x) )   = \lim_{n\to +\infty} d ( \Phi_{s_{1}} \Phi_{s^{'}_{n}}(y)  , \Phi_{s_{1}} \Phi_{s^{'}_{n}}(z) )   \leq  1/k{_{0}}<r/4$.\\
 
 (4)  From (2) and  (3,)  we obtain that $d ( \Phi_{s_{1}} (x)  , \Phi_{s_{1}} \Phi_{s_{0}} (z) )<r/2$.\\

 As $\Phi_{s_{1}}\Phi_{s_{0}}(z)\in B(w,r/8)$, from (4) we have that $\Phi_{s_{1}} (x)\in B(w,r)$.

$\Box$

\vspace{.5cm}

Let $M\subset X$ be a subset of $X$. We define $M^{-1}:=\cup_{s\in \mathcal{S}}\Phi^{-1}_{s}(M) \cup M$, where $\Phi^{-1}_{s}(M)=\{x\in X: \ \Phi_{s}(x)\in M \}$.

\begin{prop}\label{prop2}
 Let $X$ be a locally compact metric space with a countable base and $( \mathcal{S}, X, \Phi )$ that satisfy
the following two conditions:\\
(1) $Trans(X)^{*}\neq\emptyset$,\\
(2) There exists an invariant no dense set $M$ such that $M^{-1}$ is dense in $X$.\\
Then $( \mathcal{S}, X, \Phi )$  is sensitive.
\end{prop}
\begin{proof}
As $M$ is no dense, there exists $w\in X$ and $r>0$ such that $M\cap B(w,r)=\emptyset$. We will show that the dynamical system $( \mathcal{S}, X, \Phi )$ is sensitive with constant $\alpha =r/4$. Let $x\in X$ and $U$ a neighbourhood of $x$. 
As $M^{-1}$ is dense in $X$ there exists $y\in M^{-1}\cap U$. Then there exists $s_0\in \mathcal{S}$ that $\Phi_{s_{0}}(y)\in M$. Let $z\in Trans(X)^{*}\cap U $. As $z\in Trans(X)^{*}$ there exist $s_1 \in \mathcal{S}$ that $\Phi_{s_{1}}\Phi_{s_{0}}(z)\in B(w,r/2)$. As  $\Phi_{s_{1}}\Phi_{s_{0}}(y)\in M$, then $d(   \Phi_{s_{1}}\Phi_{s_{0}}(y)   ,  \Phi_{s_{1}}\Phi_{s_{0}}(z) )> r/2>\alpha $.

\end{proof}

%Recall that $\mathcal{M}$ is the set of minimal points.\\ 

%\begin{prop}\label{prop4}
% Let $X$ be a locally compact metric space with a countable base and $( \mathcal{S}, X, \Phi )$ satisfy
%the following two conditions:\\
%(1) $Trans(X)^{*}\neq\emptyset$;\\
%(2)  $\mathcal{M}^{-1}$ is dense in $X$.
%Then $( \mathcal{S}, X, \Phi )$  is sensitive to any metric $d$ metrizing $X$.
%  \end{prop}

{\bf{Proof of Theorem \ref{theo2}:}}\\
If $\mathcal{M}$ is dense, by Theorem \ref{theo0},   $( \mathcal{S}, X, \Phi )$  is sensitive.\\
If $\mathcal{M}$ is no dense, by Proposition \ref{prop2},   $( \mathcal{S}, X, \Phi )$  is sensitive.\\

$\Box$

\vspace{.5cm}
The following example shows that in Proposition~\ref{prop2}, the hypothesis \( Trans(X)^* \neq \emptyset \)  cannot be replaced by the assumption that the system is topologically transitive (TT).

\begin{figure}[h]
\psfrag{0}{$0$}
\psfrag{1}{$1$}
\psfrag{13}{$\frac{1}{3}$}
\psfrag{23}{$\frac{2}{3}$}
\psfrag{f}{$f_1\equiv 0$}
\psfrag{g}{$f_2$}
\psfrag{12}{$\frac{1}{2}$}
\psfrag{14}{$\frac{1}{4}$}
\psfrag{18}{$\frac{1}{8}$}
\psfrag{h}{$f_3$}
\begin{center}
\caption{\label{figura101}}
\subfigure[]{\includegraphics[scale=0.22]{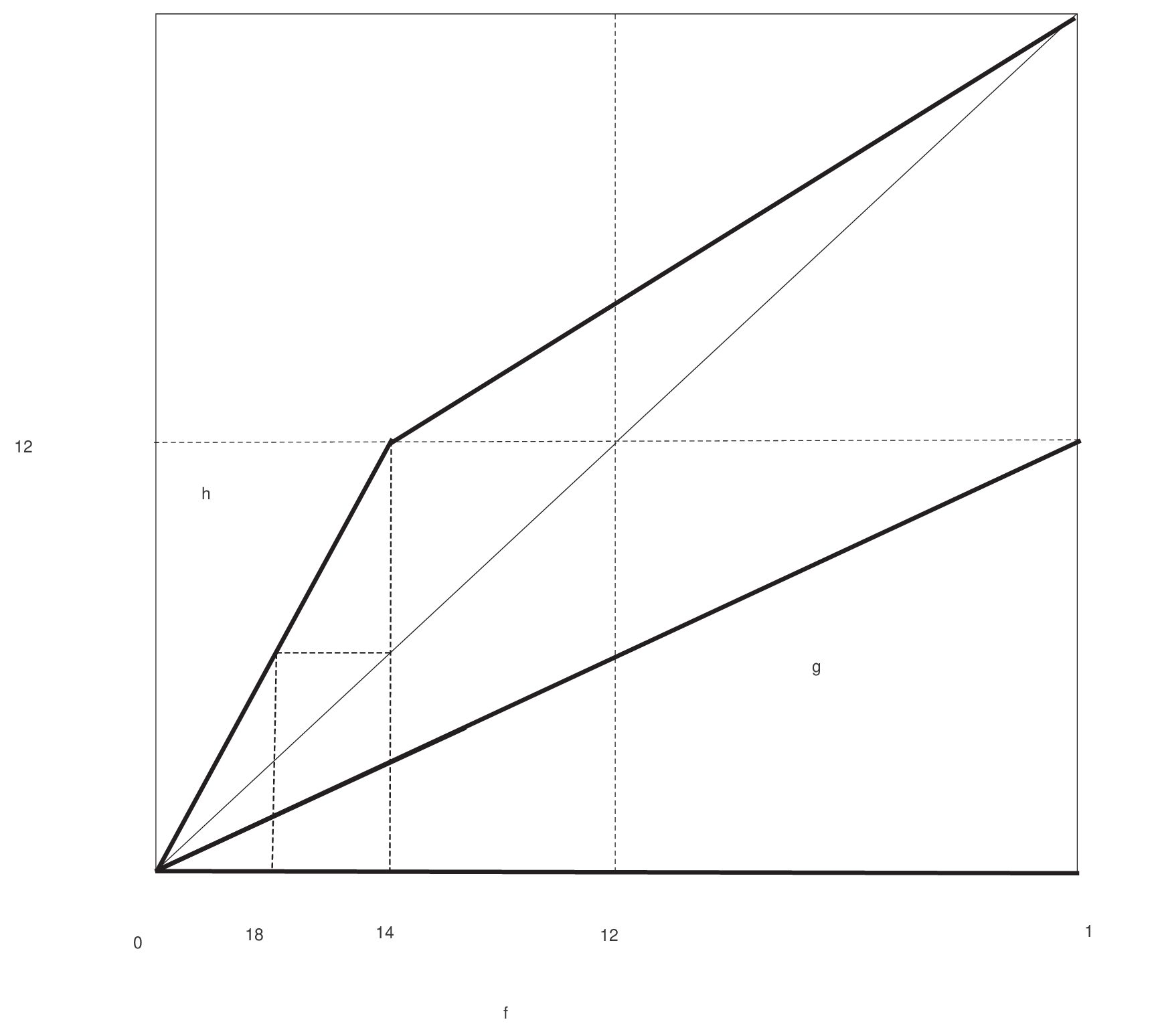}}
%\subfigure[]{\includegraphics[scale=0.16]{figura10.eps}}
\end{center}
{ }
\end{figure}

\vspace{1cm}

Let $f_1,f_2,f_3:[0,1]\to [0,1]$ be as in Figure \ref{figura101} and  $\mathcal{S}$ the semigroup generated by $f_1,f_2 $ and $f_3$. The following was proved in \cite{ip1}:
\begin{itemize}
\item $\mathcal{M}=\{0\}$ and  $f_1^{-1}(0)=[0,1]=X$, therefore $\mathcal{M}^{-1}$ is dense.
\item $(\mathcal{S},X,\Phi)$ is TT and non-minimal. 
\item $(\mathcal{S},X,\Phi)$ is non-sensitive.
\end{itemize}

\end{document}